\documentclass{amsart}
\usepackage{latexsym,amsmath,amssymb}
\usepackage{cite}
\usepackage{tikz}
\usepackage{caption}
\usepackage{tikz-3dplot}

\newtheorem{thm}{Theorem}[section]
\newtheorem{cor}[thm]{Corollary}
\newtheorem{exam}[thm]{Example}

\newtheorem{prop}[thm]{Proposition}
\theoremstyle{definition}\newtheorem{defn}[thm]{Definition}
\theoremstyle{remark}
\newtheorem{rem}[thm]{Remark}
\numberwithin{equation}{section}

\newcommand{\A}{\mathcal{A}}

\begin{document}

\title[]
{Weighted conditional expectation operators and nuclearity}

\author{\sc\bf M. S. Al Ghafri, S. Shamsigamchi and Y. Estaremi }
\address{\sc M. S. Al Ghafri, S. Shamsigamchi and Y. Estaremi}
\email{mohammed.alghafri@utas.edu.om}\email{S.shamsi@pnu.ac.ir}\email{y.estaremi@gu.ac.ir}
\address{Department of Mathematics, University of Technology and Applied Sciences, Rustaq {329},  Oman.}
\address{Department of Mathematics, Payame Noor(PNU) university, Tehran, Iran.}
\address{Department of Mathematics and Computer Sciences,
Golestan University, Gorgan, Iran.}

\thanks{}

\thanks{}



\date{}

\dedicatory{}

\commby{}

\begin{abstract}
We provide a characterisations of nuclear weighted conditional expectation operators on $L^p(\mu)$-spaces, for $1\leq p<\infty$. As a consequence, when the underlying measure space is non-atomic, the only nuclear weighted conditional expectation operator on $L^p(\mu)$-spaces is the zero operator.
\end{abstract}

\maketitle

\section{ \sc\bf Introduction and Preliminaries}
Let $(X,\Sigma,\mu)$ be a complete $\sigma$-finite measure space.
For any sub-$\sigma$-finite algebra $\mathcal{A}\subseteq
 \Sigma$, the $L^p$-space
$L^p(X,\mathcal{A},\mu_{\mid_{\mathcal{A}}})$ is abbreviated  by
$L^p(\mathcal{A})$ and its norm is denoted by $\|.\|_p$, $(1\leq p\leq \infty)$. All
comparisons between two functions or two sets are to be
interpreted as holding up to a $\mu$-null set. The support of a
measurable function $f$ is defined as $S(f)=\{x\in X; f(x)\neq
0\}$. We denote the vector space of all equivalence classes of
almost everywhere finite valued measurable functions on $X$ by
$L^0(\Sigma)$.

\vspace*{0.3cm} For a $\sigma$-finite subalgebra
$\mathcal{A}\subseteq\Sigma$, the conditional expectation operator
associated with $\mathcal{A}$ is the mapping $f\rightarrow
E^{\mathcal{A}}f$, defined for all non-negative measurable
functions $f$ as well as for all $f\in L^p(\Sigma)$, $(1\leq p\leq \infty)$, where
$E^{\mathcal{A}}f$, by the Radon-Nikodym theorem, is the unique
$\mathcal{A}$-measurable function satisfying
$\int_{A}fd\mu=\int_{A}E^{\mathcal{A}}fd\mu, \ \ \ \forall A\in
\mathcal{A}$. As an operator on $L^{p}({\Sigma})$,
$E^{\mathcal{A}}$ is idempotent and
$E^{\mathcal{A}}(L^p(\Sigma))=L^p(\mathcal{A})$. This operator
will play a major role in our work. Let $f\in L^0(\Sigma)$, then
$f$ is said to be conditionable with respect to $E$ if
$f\in\mathcal{D}(E):=\{g\in L^0(\Sigma): E(|g|)\in
L^0(\mathcal{A})\}$. We
write $E(f)$ in place of $E^{\mathcal{A}}(f)$.\\
Here we recall some basic properties of the conditional expectation operator $E$ on Hilbert space $L^p(\mu)$. Let $f,g\in L^p(\mu)$. Then\\
\vspace*{0.2cm} \noindent $\bullet$ \  If $g$ is
$\mathcal{A}$-measurable, then $E(fg)=E(f)g$.

\noindent $\bullet$ \ $|E(f)|^p\leq E(|f|^p)$.

\noindent $\bullet$ \ If $f\geq 0$, then $E(f)\geq 0$; if $f>0$,
then $E(f)>0$.

\noindent $\bullet$ \ $|E(fg)|\leq
E(|f|^p)|^{\frac{1}{p}}E(|g|^{p'})|^{\frac{1}{p'}}$, where
$\frac{1}{p}+\frac{1}{p'}=1$ (H\"{o}lder inequality).

\noindent $\bullet$ \ For each $f\geq 0$, $S(f)\subseteq
S(E(f))$. \\

 A detailed discussion and verification of
most of these properties may be found in \cite{rao}. We recall that
an $\mathcal{A}$-atom of the measure $\mu$ is an element
$A\in\mathcal{A}$ with $\mu(A)>0$ such that for each
$F\in\mathcal{A}$, if $F\subseteq A$, then either $\mu(F)=0$ or
$\mu(F)=\mu(A)$. A measure space $(X,\Sigma,\mu)$ with no atoms is
called non-atomic measure space. It is well-known fact that every
$\sigma$-finite measure space $(X,
\mathcal{A},\mu_{\mid_{\mathcal{A}}})$ can be partitioned uniquely
as $X=\left (\bigcup_{n\in\mathbb{N}}A_n\right )\cup B$, where
$\{A_n\}_{n\in\mathbb{N}}$ is a countable collection of pairwise
disjoint $\mathcal{A}$-atoms and $B$, being disjoint from each
$A_n$, is non-atomic (see \cite{z}).

It is easy to see that if the Sigma subalgebra $\mathcal{A}$ is generated by a sequence of measurable sets $\{A_n\}_{n\in \mathbb{N}}$ with positive measures, then for every $f\in\mathcal{D}(E)$,
$$E(f)=\sum_{n=1}^{\infty}\left(\frac{1}{\mu(A_n)}\int_{A_n}fd\mu\right).\chi_{A_n}$$

\begin{defn}
Let $(X,\Sigma,\mu)$ be a $\sigma$-finite measure space and let $\mathcal{A}$ be a
$\sigma$-subalgebra of $\Sigma$, such that $(X,\mathcal{A},\mu)$ is also $\sigma$-finite. Suppose that $E$ is the corresponding conditional
expectation operator on $L^p(\mu)$ relative to $\mathcal{A}$. If $1\leq p\leq\infty$, $w,u \in L^0(\mu)$ such that $uf$ is conditionable, for all $f\in L^p(\mu)$, then the corresponding weighted conditional expectation(WCE) operator is the linear
transformation $M_wEM_u:L^p(\mu)\rightarrow L^0(\mu)$ defined by $f\rightarrow wE(uf)$.
\end{defn}
Interested readers can find more information about WCE operators in \cite{ej,her,dhd}.\\
Many operator-theoretic properties of WCE operators on $L^p$-spaces were studied in \cite{e2,e,e1,ej,her,dhd} and such operators acting between distinct $L^p$-spaces in \cite{ej, ej2}. In this paper we characterize nuclear WCE operators on $L^p(\mu)$-spaces, for $1\leq p<\infty$.

\section{ \sc\bf Main Results}
Let $X$ and $Y$ be Banach spaces and $X^*$ be the dual of $X$.
A bounded operator $T:X\to Y$ is \emph{nuclear} if it admits a representation
\begin{equation}\label{A}
Tx=\sum_{n=1}^{\infty} f_n(x)\,y_n,
\qquad x\in X,
\end{equation}
for some $(f_n)\subset X^*$ and $(y_n)\subset Y$ with
$\sum_{n=1}^{\infty}\|f_n\|_{X^*}\,\|y_n\|_{Y}<\infty$.
The \emph{nuclear norm} is
\[
\|T\|_{N}
:=\inf\left\{\sum_{n=1}^{\infty}\|f_n\|_{X^*}\,\|y_n\|_{Y}:\ T=\sum_{n=1}^{\infty} f_n\otimes y_n\right\},
\]
where $(f\otimes y)(x):=f(x)\,y$. We refer to \cite{pi,die} for background on nuclear and summing operators,
including the ideal properties and standard consequences used implicitly below.

\medskip
An operator $T:X\to Y$ is \emph{absolutely $1$--summing} if there exists $M>0$ such that for every finite family
$x_1,\dots,x_m\in X$,
\begin{equation}\label{eq:1summing}
\sum_{i=1}^{m}\|T(x_i)\|
\le M\,\sup_{f\in X^*,\,\|f\|\le 1}\sum_{i=1}^{m}|f(x_i)|.
\end{equation}
The least such $M$ is denoted by $\pi_1(T)$. Using the weak $\ell_1$--norm identification, \eqref{eq:1summing} is equivalent to
\[
\sum_{i=1}^{m}\|T(x_i)\|
\le M\,\sup_{|\alpha_i|=1}\left\|\sum_{i=1}^{m}\alpha_i x_i\right\|.
\]

In the following we recall the Pietsch's domination theorem that will be used in the proof of our main results.

\begin{thm}[Pietsch domination theorem {\cite[Theorem~2.12]{die}}]\label{thm:Pietsch}
Let $1\le p<\infty$ and let $T:X\to Y$ be absolutely $p$--summing.
Then there exist $C>0$ and a Borel probability measure $\mu$ on the weak$^{*}$--compact unit ball $B_{X^*}$ such that
\[
\|Tx\|
\le C\left(\int_{B_{X^*}} |x^*(x)|^p\,d\mu(x^*)\right)^{1/p},
\qquad x\in X.
\]
In particular, for $p=1$,
\[
\|Tx\|\le C\int_{B_{X^*}} |x^*(x)|\,d\mu(x^*).
\]
\end{thm}

As we have proved in \cite{ej}, the WCE operator $T=M_wEM_u$ is a bounded operator on $L^{p}(\Sigma)$ if and only if
$(E|w|^{p})^{\frac{1}{p}}(E|u|^{p'})^{\frac{1}{p'}} \in
L^{\infty}(\mathcal{A})$,
 and in this case its norm is as follows
$$\|T\|=\|(E(|w|^{p}))^{\frac{1}{p}}(E(|u|^{p'}))^{\frac{1}{p'}}\|_{\infty},$$
in which $1<p<\infty$ and $\frac{1}{p}+\frac{1}{p'}=1$.\\
Also, the WCE operator $T$ is bounded on $L^{1}(\Sigma)$ if and only if
$uE(|w|) \in L^{\infty}(\Sigma)$ and in this case
$\|T\|=\|uE(|w|)\|_{\infty}$.\\

Moreover, in \cite{ej} it is shown that for $1<p<\infty$, WCE operator $M_wEM_u:L^{p}(\Sigma)\rightarrow
L^{p}(\Sigma)$ is compact if and only if for each $\varepsilon>0$
the set
 $$N_{\varepsilon}=\{x\in
X:(E(|u|^{p'}))^{\frac{1}{p'}}(x)(E(|w|^{p}))^{\frac{1}{p}}(x)\geq\varepsilon\},$$
consists of finitely many $\mathcal{A}$-atoms. And also, it is proved that the WCE operator $T=M_wEM_u:L^{1}(\Sigma)\rightarrow L^{1}(\Sigma)$ is
compact if and only if for each $\varepsilon>0$ the set
 $$K_\varepsilon:=\{x\in X:u(x)E(|w|)(x)\geq\varepsilon\}$$
 consists of finitely many
$\Sigma$-atoms.

\begin{cor}\label{c2.2} Let $1\leq p<\infty$. If the
measure space $(X,\mathcal{A},\mu)$ is non-atomic, then the WCE operator $T=M_wEM_u:L^{p}(\Sigma)\rightarrow L^{p}(\Sigma)$ is
nuclear if and only if  $T=0$.
\end{cor}
\begin{proof} Let $T=M_wEM_u$ be a nonzero
nuclear WCE operator. Then $T$ is compact and for the case $1<p<\infty$,
$$0\neq\|T\|=\|(E(|w|^{p}))^{\frac{1}{p}}(E(|u|^{p'}))^{\frac{1}{p'}}\|_{\infty},$$
and so there exists $\delta>0$ such that the set
$$\{x\in
X:(E(|w|^{p}))^{\frac{1}{p}}(x)(E(|u|^{p'}))^{\frac{1}{p'}}(x)>\delta\},$$
contains an $\mathcal{A}$-measurable subset $B$ with positive
measure. Hence, we obtain a sequence of pairwise disjoint subsets
$B_{n}\subseteq B$ such that for every $n\in\mathbb{N}$,
$B_{n}\in\A$ and $0<\mu(B_{n})<\infty$. Define
$$f_{n}=\frac{\overline{u}|u|^{p'-2}(E(|w|^{p}))^{\frac{p'-1}{p}}}
{\|T\|^{\frac{p'}{p}} (\mu(B_{n}))^{\frac{1}{p}}}\chi_{B_{n}}.$$
It is easy to see that $\|f_{n}\|_{p}\leq1$ and
\begin{align*}
\|Tf_{n}-Tf_{m}\|^{p}_{p}&=\int_{X}|w|^{p}|E(f_{n}-f_{m})|^{p}d\mu\\
&\geq\int_{B_{n}}\frac{(E(|w|^{p}))^{p'}(E(|u|^{p'}))^{p}}{\|T\|^{p'}\mu(B_{n})}d\mu\\
&+ \int_{B_{m}}\frac{(E(|w|^{p}))^{p'}(E(|u|^{p'}))^{p}}{\|T\|^{p'}\mu(B_{m})}d\mu\\
 &\geq \frac{2\delta^{pp'}} {\|T\|^{p'}}.
 \end{align*}
 Also, for the case $p=1$, we have
 $0\neq\|T\|=\|uE(|w|)\|_{\infty}$ and so $\mu(S(uE(|w|))>0$. Also, we have $S(uE(|w|)\subseteq S(E(|u|)E(|w|)$ and $\|E(|u|)E(|w|)\|_{\infty}\leq \|uE(|w|)\|_{\infty}<\infty$. Hence $\mu(S(E(|u|)E(|w|))>0$ and then there exists $\delta>0$ such that the set
$$\{x\in
X:E(|u|)(x)E(|w|)(x)>\delta\},$$
contains an $\mathcal{A}$-measurable subset $B$ with positive
measure. Hence, we obtain a sequence of pairwise disjoint subsets
$B_{n}\subseteq B$ such that for every $n\in\mathbb{N}$,
$B_{n}\in\A$ and $0<\mu(B_{n})<\infty$. Define
$$f_{n}=\frac{\overline{E(u)}E(|w|)}
{(\|T\|^2+1)\mu(B_{n})}\chi_{B_{n}}.$$
It is easy to see that $\|f_{n}\|_{1}\leq1$ and $\{Tf_n\}_{n\in \mathbb{N}}$ has no convergent subsequence.

  This implies that in both cases, $T$ is not compact. But this is a contradiction. The converse is obvious.\\
\end{proof}

\begin{cor}\label{c2.4}
If the measure space $(X, \mathcal{A}, \mu)$, with finite subset property, is not purely atomic, such that the restriction of the WCE operator $T=M_wEM_u$ to the non-atomic part is a non-zero operator, then $T$ can not be nuclear on $L^p(\Sigma)$ for $1\leq p<\infty$.
\end{cor}
\begin{proof}
Similar to the proof of Corollary \ref{c2.2}, if there is a non-atomic measurable set $A\in \mathcal{A}$, with $\mu(A)>0$, then we can find a bounded sequence $\{f_n\}$ in
$L^p(\mathcal{A})$ such that the sequence $\{Tf_n\}$ has no convergent subsequence. So $T$ is not compact and therefore $T$ is not nuclear on $L^p(\Sigma)$.
\end{proof}

\begin{prop}\label{p2.5}
Let $1<p<\infty$ and $\mathcal{A}\subseteq \Sigma$ be a $\sigma$- finite subalgebra such that the measure space $(X, \mathcal{A}, \mu\mid_{\mathcal{A}})$ is a purely atomic measure space. Then the bounded WCE operator $T=M_wEM_u$ is nuclear on $L^p(\Sigma)$ if and only if
$$\sum^{\infty}_{n=1}(E(|w|^{p})(A_n))^{\frac{1}{p}}(E(|u|^{q})(A_n))^{\frac{1}{q}}<\infty,$$
in which $\{A_n\}^{\infty}_{n=1}$ pair-wise disjoint $\mathcal{A}$-atoms such that $X=\cup^{\infty}_{n=0}A_n$ and $\frac{1}{p}+\frac{1}{q}=1$.
\end{prop}
\begin{proof} Since the measure space $(X, \mathcal{A}, \mu\mid_{\mathcal{A}})$ is a $\sigma$-finite and purely atomic, then there exists a sequence of pair-wise disjoint $\mathcal{A}$-atoms $\{A_n\}^{\infty}_{n=1}$ such that $X=\cup^{\infty}_{n=0}A_n$. As is known in the literature the $\mathcal{A}$-measurable functions defined on $X$ are fixed on each $\mathcal{A}$-atom. Hence for every $f\in L^p(\Sigma)$ we have
\begin{align*}
Tf=w.E(uf)&=w.\sum_{n=1}^{\infty}E(u.f)(A_n).\chi_{A_n}\\
&=w.\sum_{n=1}^{\infty}\left(\frac{1}{\mu(A_n)}\int_{A_n}E(uf)d\mu\right).\chi_{A_n}\\
&=w.\sum_{n=1}^{\infty}\left(\frac{1}{\mu(A_n)}\int_{A_n}ufd\mu\right).\chi_{A_n}.
\end{align*}
Let $$\sum^{\infty}_{n=1}(E(|w|^{p})(A_n))^{\frac{1}{p}}(E(|u|^{q})(A_n))^{\frac{1}{q}}<\infty.$$
 Since $T$ is bounded and the measure space $(X, \mathcal{A}, \mu\mid_{\mathcal{A}})$ is a purely atomic, then we have
$$\|T\|=\|(E(|w|^{p}))^{\frac{1}{p}}(E(|u|^{q}))^{\frac{1}{q}}\|_{\infty}=\sup_{n}(E(|w|^{p})(A_n))^{\frac{1}{p}}(E(|u|^{q})(A_n))^{\frac{1}{q}}<\infty.$$
If we set $\phi_n(f)=\int_X(u.\chi_{A_n})fd\mu$ and $g_n=\frac{w.\chi_{A_n}}{\mu(A_n)}$, then we get that for every $n\in \mathbb{N}$, $u.\chi_{A_n}\in L^{q}(\Sigma)$ and  $g_n\in L^{p}(\Sigma)$. It is clear that for each $n\in \mathbb{N}$, the function $\phi_n$ is a bounded linear functional on $L^p(\Sigma)$. Also, $\|g_n\|_p=(E(|w|^p)(A_n))^{\frac{1}{p}}$ and $\|\phi_n\|=(E(|u|^{q})(A_n))^{\frac{1}{q}}$. Therefore,

$$T=M_wEM_u=\sum_{n=1}^{\infty}\phi_n\otimes g_n,$$
and
$$\sum_{n=1}^{\infty}\|\phi_n\|\|g_n\|=\sum^{\infty}_{n=1}(E(|w|^{p})(A_n))^{\frac{1}{p}}(E(|u|^{q})(A_n))^{\frac{1}{q}}<\infty.$$
This implies that $T$ is nuclear.\\

Conversely, let the WCE operator $T=M_wEM_u$ be nuclear. Then $T$ is absolutely summing and so by Pietsch’s  theorem 2.12 of \cite{die}, there is a Borel probability measure $\nu$ on the unit ball $B_q$ of $L^q(\mu)\simeq(L^p(\mu))^*$ such that for some $M>0$

$$\|Tf\|_p\leq M\int_{B_q}|l(f)|d\nu(l), \ \ \ \ \ \forall f\in L^p(\mu).$$
By Riesz representation theorem, for each $l\in (L^p(\mu))^*$, there exists $g_l\in L^q(\mu)$ such that
  $$l(f)=\int_Xf.g_l d\mu, \ \ \ \ \ f\in L^p(\mu).$$
  
So by Holder's inequality we have 
\begin{align*}
\|Tf\|_p&\leq M\int_{B_q}|l(f)|d\nu(l)\\
&=M\int_{B_q}\int_X|f.g_l| d\mu d\nu(g_l)\\
&\leq M\int_{B_q}\|g_l\|_q\|f\|_pd\nu(g_l),
\end{align*}
 
Since $\{A_n\}^{\infty}_{n=1}$ is a sequence of disjoint atoms, for each $i\in \mathbb{N}$ we define
$$f_i=\frac{\overline{u}|u|^{q-2}(E(|w|^{p}))^{\frac{q-1}{p}}}{\|T\|^{\frac{q}{p}} (\mu(A_{i}))^{\frac{1}{p}}}\chi_{A_{i}}.$$
It is easy to see that $f_i\in L^p(\Sigma)$ and $\|f_i\|_p\leq 1$. Hence we have
\begin{align*}
\|Tf_i\|^p_p&=\frac{1}{\|T\|^{q}}\int_{A_i}\frac{1}{\mu(A_{i})}\left((E(|w|^{p}))^{\frac{1}{p}}(E(|u|^{q}))^{\frac{1}{q}}\right)^{pq}d\mu\\
&=\frac{1}{\|T\|^{q}}\left((E(|w|^{p})(A_i))^{\frac{1}{p}}(E(|u|^{q})(A_i))^{\frac{1}{q}}\right)^{pq}.
\end{align*}
By these observations we get 
\begin{align*}
\|Tf_i\|_p&=\frac{1}{\|T\|^{\frac{q}{p}}}\left((E(|w|^{p})(A_i))^{\frac{1}{p}}(E(|u|^{q})(A_i))^{\frac{1}{q}}\right)^{q}\\
&\leq M\int_{B_q}\|g_l.\chi_{A_i}\|_q\|f_i\|_pd\nu(g_l)\\
&\leq  M\int_{B_q}\|g_l.\chi_{A_i}\|_qd\nu(g_l).\\
\end{align*}
And so 
$$(E(|w|^{p})(A_i))^{\frac{1}{p}}(E(|u|^{q})(A_i))^{\frac{1}{q}}\leq M^{\frac{1}{q}}\|T\|^{\frac{1}{p}}\left(\int_{B_q}\|g_l.\chi_{A_i}\|_qd\nu(g_l)\right)^{\frac{1}{q}}.$$

It is clear that $\|g_l\|_q=\sum^{\infty}_{i=1}\|g_l.\chi_{A_i}\|_q$. Therefore
\begin{align*}
\sum^{\infty}_{i=1}(E(|w|^{p})(A_i))^{\frac{1}{p}}(E(|u|^{q})(A_i))^{\frac{1}{q}}&\leq M^{\frac{1}{q}}\|T\|^{\frac{1}{p}}\sum^{\infty}_{i=1}\left(\int_{B_q}\|g_l.\chi_{A_i}\|_qd\nu(g_l)\right)^{\frac{1}{q}}\\
&\leq M^{\frac{1}{q}}\|T\|^{\frac{1}{p}}\left(\int_{B_q}\sum^{\infty}_{i=1}\|g_l.\chi_{A_i}\|_qd\nu(g_l)\right)^{\frac{1}{q}}\\
&=M^{\frac{1}{q}}\|T\|^{\frac{1}{p}}\left(\int_{B_q}\|g_l\|_qd\nu(g_l)\right)^{\frac{1}{q}}\\
&\leq M^{\frac{1}{q}}\|T\|^{\frac{1}{p}}.
\end{align*}
This implies that 

$$\sum^{\infty}_{i=1}(E(|w|^{p})(A_i))^{\frac{1}{p}}(E(|u|^{q})(A_i))^{\frac{1}{q}}\leq M^{\frac{1}{q}}\|T\|^{\frac{1}{p}}<\infty.$$

This completes the proof.
\end{proof}

In the following theorem we characterize compact WCE operators on $L^1(\Sigma)$.
\begin{thm}
The WCE operator $T:L^{1}(\Sigma)\rightarrow L^{1}(\Sigma)$ is
compact if and only if for each $\varepsilon>0$ the set
$$K_\varepsilon:=\{x\in X:E(u)(x)E(|w|)(x)\geq\varepsilon\}$$
consists of finitely many $\mathcal{A}$-atoms.
\end{thm}
\begin{proof}
Suppose that $T$ is
compact but for some
$\varepsilon>0$ the set $K_\varepsilon$ either contains a
 subset of non-atomic part $A$ with positive measure or has infinitely many $\mathcal{A}$-atoms. In both cases we
 can find a sequence of pairwise disjoint measurable subsets
 $\{A_{n}\}_{n\in \mathbb{N}}$ with $0<\mu(A_{n})<\infty$.
 Define $f_{n}=\frac{\overline{E(u)}E(|w|)\chi_{A_{n}}}{\mu(A_{n})\|T\|^2}$. Then
 $\|f_{n}\|_{1}\leq1/\|T\|$ and $\sigma(f_n)\cap\sigma(f_m)=\emptyset$, for $n\neq
 m$. Hence we have
  Then
  \begin{align*}
  \|Tf_{n}-Tf_{m}\|_{1}&\geq\int_{A_{n}}E(|w|)|E(uf_{n}-uf_{m})|d\mu\\
 &=\int_{A_{n}}\frac{(E(|w|))^2|E(u)|^2\chi_{A_{n}}}{\mu(A_{n})\|T\|^2}d\mu\\
 &\geq\frac{\varepsilon^{2}}{\|T\|^2}.
 \end{align*}
which shows that the sequence $\{Tf_n\}$ dose not contain a
convergent subsequence, and so $T$ is not compact.\\

Conversely, suppose that for each $\varepsilon>0$,
$K_\varepsilon=\cup^{n}_ {k=1}A^{\varepsilon}_{k}$. Define
$T_{\varepsilon}(f)=T(f\chi_{K_\varepsilon})=\Sigma_{k=1}^nT(f\chi_{A^{\varepsilon}_{k}})$,
for all $f\in L^{1}(\Sigma)$. It is clear that $T_{\varepsilon}$ is
a finitely rank operator on $L^{1}(\Sigma)$ and
$\|T-T_{\varepsilon}\|<\varepsilon$. Thus $T$ is compact.
\end{proof}

\begin{prop}\label{p2.7}
Let $\mathcal{A}\subseteq \Sigma$ be a $\sigma$- finite subalgebra such that the measure space $(X, \mathcal{A}, \mu)$ is a purely atomic measure space. Then the bounded WCE operator $T=M_wEM_u:L^1(\Sigma)\rightarrow L^1(\Sigma)$ is nuclear on if and only if
$$\sum^{\infty}_{n=1}E(|w|)(A_n)E(|u|)(A_n)<\infty,$$
in which $\{A_n\}^{\infty}_{n=1}$ pair-wise disjoint $\mathcal{A}$-atoms such that $X=\cup^{\infty}_{n=0}A_n$.
\end{prop}
\begin{proof}
Slimilar to the proof of Proposition \ref{p2.5} for every $f\in L^p(\Sigma)$ we have
\begin{align*}
Tf=w.E(uf)&=w.\sum_{n=1}^{\infty}E(u.f)(A_n).\chi_{A_n}\\
&=w.\sum_{n=1}^{\infty}\left(\frac{1}{\mu(A_n)}\int_{A_n}E(uf)d\mu\right).\chi_{A_n}\\
&=w.\sum_{n=1}^{\infty}\left(\frac{1}{\mu(A_n)}\int_{A_n}ufd\mu\right).\chi_{A_n}.
\end{align*}
Let
$$\sum^{\infty}_{n=1}E(|w|)(A_n)E(|u|)(A_n)<\infty.$$
 Since $T$ is bounded and the measure space $(X, \mathcal{A}, \mu\mid_{\mathcal{A}})$ is a purely atomic, then we have
$$\sup_{n\in \mathbb{N}}\left[E(|w|)(A_n)E(|u|)(A_n)\right]=\|E(|w|)E(|u|)\|_{\infty}\leq\|uE(|w|)\|_{\infty}=\|T\|.$$
If we set $\phi_n(f)=\int_X(u.\chi_{A_n})fd\mu$ and $g_n=\frac{w.\chi_{A_n}}{\mu(A_n)}$, then we get that for every $n\in \mathbb{N}$, $u.\chi_{A_n}\in L^{\infty}(\Sigma)$ and  $g_n\in L^{1}(\Sigma)$. It is clear that for each $n\in \mathbb{N}$, the function $\phi_n$ is a bounded linear functional on $L^1(\Sigma)$. Also, $\|g_n\|_1=E(|w|)(A_n)$ and $\|\phi_n\|\leq E(|u|)(A_n)$. Therefore,

$$T=M_wEM_u=\sum_{n=1}^{\infty}\phi_n\otimes g_n,$$
and
$$\sum_{n=1}^{\infty}\|\phi_n\|\|g_n\|\leq\sum^{\infty}_{n=1}E(|w|)(A_n)E(|u|)(A_n)<\infty.$$
This implies that $T$ is nuclear.\\

Conversely, let the WCE operator $T=M_wEM_u$ be nuclear on $L^1(\mu)$. Then $T$ is absolutely summing and so by Pietsch’s  theorem 2.12 of \cite{die}, there is a Borel probability measure $\nu$ on the unit ball $B_{\infty}$ of $L^{\infty}(\mu)\simeq(L^1(\mu))^*$ such that for some $M>0$

$$\|Tf\|_1\leq M\int_{B_{\infty}}|l(f)|d\nu(l), \ \ \ \ \ \forall f\in L^1(\mu).$$
By Riesz representation theorem, for each $l\in (L^1(\mu))^*$, there exists $g_l\in L^{\infty}(\mu)$ such that
  $$l(f)=\int_Xf.g_l d\mu, \ \ \ \ \ f\in L^1(\mu).$$
  
And hence we have
\begin{align*}
\|Tf\|_1&\leq M\int_{B_{\infty}}|l(f)|d\nu(l)\\
&=M\int_{B_{\infty}}\int_X|f.g_l| d\mu d\nu(g_l)\\
&\leq M\int_{B_{\infty}}\|g_l\|_{\infty}\|f\|_1d\nu(g_l),
\end{align*}
 
Since $\{A_n\}^{\infty}_{n=1}$ is a sequence of disjoint atoms, for each $n\in \mathbb{N}$ we define
$$f_{n}=\frac{\overline{E(u)}E(|w|)\chi_{A_{n}}}{\mu(A_{n})\|T\|(\|T\|+1)}.$$
It is easy to see that $f_n\in L^1(\Sigma)$ and $\|f_n\|_1\leq 1$. Hence we have

\begin{align*}
\|Tf_n\|_1&=\frac{1}{\|T\|(\|T\|+1)}\int_{A_n}\frac{1}{\mu(A_{n})}\left(E(|w|)E(|u|)\right)^{2}d\mu\\
&=\frac{1}{\|T\|(\|T\|+1)}\left(E(|w|)(A_n)E(|u|)(A_n)\right)^{2}.
\end{align*}
By these observations we get that
\begin{align*}
\|Tf_n\|_1&=\frac{1}{\|T\|(\|T\|+1)}\left(E(|w|)(A_n)E(|u|)(A_n)\right)^{2}\\
&\leq M\int_{B_{\infty}}\|g_l.\chi_{A_n}\|_{\infty}\|f\|_1d\nu(g_l),\\
&\leq  M\int_{B_{\infty}}\|g_l.\chi_{A_n}\|_{\infty}d\nu(g_l).\\
\end{align*}
And so
$$E(|w|)(A_n)E(|u|)(A_n)\leq  M^{\frac{1}{2}}(\|T\|(\|T\|+1))^{\frac{1}{2}}\left(\int_{B_{\infty}}\|g_l.\chi_{A_n}\|_{\infty}d\nu(g_l)\right)^{\frac{1}{2}}.\\$$
It is clear that $\|g_l\|_{\infty}=\sum^{\infty}_{n=1}\|g_l.\chi_{A_n}\|_{\infty}$. Therefore
\begin{align*}
\sum^{\infty}_{i=1}E(|w|)(A_n)E(|u|)(A_n)&\leq M^{\frac{1}{2}}(\|T\|(\|T\|+1))^{\frac{1}{2}}\sum^{\infty}_{i=1}\left(\int_{B_{\infty}}\|g_l.\chi_{A_n}\|_{\infty}d\nu(g_l)\right)^{\frac{1}{2}}\\
&\leq M^{\frac{1}{2}}(\|T\|(\|T\|+1))^{\frac{1}{2}}\left(\int_{B_{\infty}}\sum^{\infty}_{i=1}\|g_l.\chi_{A_n}\|_{\infty}d\nu(g_l)\right)^{\frac{1}{2}}\\
&=M^{\frac{1}{2}}(\|T\|(\|T\|+1))^{\frac{1}{2}}\left(\int_{B_{\infty}}\|g_l\|_{\infty}d\nu(g_l)\right)^{\frac{1}{2}}\\
&\leq M^{\frac{1}{2}}(\|T\|(\|T\|+1))^{\frac{1}{2}}.
\end{align*}
This implies that 
$$\sum^{\infty}_{i=1}E(|w|)(A_n)E(|u|)(A_n)\leq M^{\frac{1}{2}}(\|T\|(\|T\|+1))^{\frac{1}{2}}<\infty.$$
\end{proof}

\begin{rem}
The WCE operator $T=M_wEM_u:L^p(\mu)\rightarrow L^p(\mu)$ is nuclear if and only if $T^*$ is nuclear, for all $1\leq p<\infty$.
\end{rem}

Let $\mathcal{A}\subseteq \Sigma$ be a $\sigma$-finite subalgebra of $\Sigma$. Then $X=\left (\bigcup_{n\in\mathbb{N}}A_n\right )\cup B$, where
$\{A_n\}_{n\in\mathbb{N}}$ is a countable collection of pairwise disjoint $\mathcal{A}$-atoms and $B$, being disjoint from each
$A_n$, is non-atomic. In the following theorem we characterize nuclear WCE operators on general $L^p(\mu)$-spaces.

  \begin{thm}\label{t2.10}
Let $1<p<\infty$ and $\mathcal{A}\subseteq \Sigma$ be a $\sigma$- finite subalgebra. Then the bounded WCE operator $T=M_wEM_u$ is nuclear on $L^p(\Sigma)$ if and only if
\begin{itemize}
  \item $(E(|w|^{p}))^{\frac{1}{p}}(E(|u|^{q})))^{\frac{1}{q}}=0$, $\mu$-a.e., on $B$. 
  \item $\sum^{\infty}_{n=1}(E(|w|^{p})(A_n))^{\frac{1}{p}}(E(|u|^{q})(A_n))^{\frac{1}{q}}<\infty$,
\end{itemize}
in which $\{A_n\}^{\infty}_{n=1}$ pair-wise disjoint $\mathcal{A}$-atoms and $B$ is non-atomic such that $X=\cup^{\infty}_{n=0}A_n\cup B$ and $\frac{1}{p}+\frac{1}{q}=1$.
\end{thm}
\begin{proof}
By combining Proposition \ref{p2.5} and Corollary \ref{c2.4} we have the result.
\end{proof}
\begin{thm}\label{t2.11}
Let $\mathcal{A}\subseteq \Sigma$ be a $\sigma$- finite subalgebra. Then the bounded WCE operator $T=M_wEM_u$ is nuclear on $L^1(\Sigma)$ if and only if
\begin{itemize}
  \item $E(|w|)E(|u|)=0$, $\mu$-a.e., on $B$.
  \item $\sum^{\infty}_{n=1}E(|w|)(A_n)E(|u|)(A_n)<\infty$,
\end{itemize}
in which $\{A_n\}^{\infty}_{n=1}$ pair-wise disjoint $\mathcal{A}$-atoms and $B$ is non-atomic such that $X=\cup^{\infty}_{n=0}A_n\cup B$ and $\frac{1}{p}+\frac{1}{q}=1$.
\end{thm}
\begin{proof}
By combining Proposition \ref{p2.7} and Corollary \ref{c2.4} we have the result.
\end{proof}
\begin{cor}
Let $\mathcal{A}\subseteq \Sigma$ be a $\sigma$- finite subalgebra and $1\leq p<\infty$. Then the multiplication operator $M_u:L^p(\mu)\rightarrow L^p(\mu)$ is nuclear if and only if 
$u=0$, $\mu$-a.e., on $B$ and 
$$\sum^{\infty}_{n=1}|u|(A_n)<\infty.$$

\end{cor}
\begin{exam}
Let $X=\mathbb{N}$,
$\Sigma=2^{\mathbb{N}}$ and let $\mu$ be the counting measure. Put
$$A=\{\{2\},\{4,6\},
\{8,10,12\},\{14,16,18,20\},\cdots\}\cup\{\{1\}, \{3\}, \{5\},
\cdots \}.$$
If we let $A_{1}=\{2\}$, $A_{2}=\{4,6\}$,
$A_{3}=\{8,10,12\}$, $\cdots$, then we see that $\mu(A_{n})=n$
and for every $n\in \mathbb{N}$, there exists $k_{n}\in
\mathbb{N}$ such that $A_{n}=\{2k_{n}, 2(k_{n}+1), \cdots,
2(k_{n}+n-1)\}$. Let $\mathcal{A}$ be the $\sigma$-algebra
generated by the partition $A$ of $\mathbb{N}$. Note that,
$\mathcal{A}$ is a sub-$\sigma$-finite algebra of $\Sigma$ and
each of element of $\mathcal{A}$ is an $\mathcal{A}$-atom. It is
known that the conditional expectation of any  $f\in
\mathcal{D}(E)$ relative to $\mathcal{A}$ is
$$E(f)=\sum_{n=1}^{\infty}\left(\frac{1}{\mu(A_n)}\int_{A_n}fd\mu\right)\chi_{A_n}+
\sum_{n=1}^{\infty}f(2n-1)\chi_{\{2n-1\}}.$$
Define $u(n)=n$ and $w(n)=\frac{1}{n^{3}}$, for all $n\in \mathbb{N}$.
For each even number $m\in \mathbb{N}$, there exists $n_{m}\in \mathbb{N}$ such that
$m\in A_{n_{m}}$. Hence for every$ 1< p<\infty$ we have
$$E(|w|^p)(m)=\frac{1}{(2k_{n_m})^{3p}}+\cdots
+\frac{1}{(2k_{n_m}+2n_m-2)^{3p}}$$
and
$$E(|u|^{q})(m)=2^qk^q_{n_m}+2^q(k_{n_m}+1)^q+ \cdots+
2^q(k_{n_m}+n_m-1)^q$$
Since $n_m\leq k_{n_m}$ we have

$$(E(|w|^p))^{\frac{1}{p}}(m)(E(|u|^{q}))^{\frac{1}{q}}(m)\leq\frac{4k_{n_m}}{2^3k^3_{n_m}}.$$

Also, for all $n\in \mathbb{N}$
$$(E(|w|^p))^{\frac{1}{p}}(2n-1)(E(|u|^{p}))^{\frac{1}{p}}(2n-1)\leq
\frac{1}{(2n-1)^2}.$$
Therefore 
\begin{align*}
\sum^{\infty}_{n=1}(E(|w|^{p})(A_n))^{\frac{1}{p}}(E(|u|^{q})(A_n))^{\frac{1}{q}}&+\sum^{\infty}_{n=1}(E(|w|^{p})(2n-1))^{\frac{1}{p}}(E(|u|^{q})(2n-1))^{\frac{1}{q}}\\
&\leq \sum^{\infty}_{n=1}\frac{1}{2k^2_{n}}+\sum^{\infty}_{n=1}\frac{1}{(2n-1)^2}\\
&<\infty,
\end{align*}
and so by Theorem \ref{t2.10} we get that the WCT operator $T=M_wEM_u$ is a nuclear operator on $L^p(\mu)$. Similarly we have
\begin{align*}
\sum^{\infty}_{n=1}E(|w|)(A_n)E(|u|)(A_n)&+\sum^{\infty}_{n=1}E(|w|)(2n-1)E(|u|)(2n-1)\\
&\leq \sum^{\infty}_{n=1}\frac{1}{2k^2_{n}}+\sum^{\infty}_{n=1}\frac{1}{(2n-1)^2}\\
&<\infty,
\end{align*}
and consequently by Theorem \ref{t2.11} we get that we get that the WCT operator $T=M_wEM_u$ is a nuclear operator on $L^1(\mu)$.

\end{exam}

{\bf Declarations}
\begin{itemize}
\item Conflicts of Interest:
The authors declare no conflicts of interest.
\item Author Contributions:
The authors have the same contribution.
\item Data Availability Statement:
Data is not applicable. The results are obtained by manual computations.
\item Funding:
No funding.
\end{itemize}

\end{document}